\documentclass[11pt,oneside,letterpaper]{amsart}
\usepackage{amssymb}
\usepackage{amsfonts,amsmath,amstext,amsbsy,euscript,amsthm}
\usepackage{mathtools}
\usepackage{enumerate}
\usepackage{mathabx}
\usepackage[all]{xy}
\usepackage[utf8]{inputenc} 
\usepackage{marginnote}
\usepackage{color}
\usepackage{tikz-cd}
\usepackage{mathrsfs}
\usepackage[colorlinks]{hyperref}
\usepackage{dsfont}
\usepackage[shortlabels]{enumitem}

\theoremstyle{plain}
\newtheorem{theorem}{Theorem}[section]
\newtheorem{proposition}[theorem]{Proposition}

\newtheorem{lemma}[theorem]{Lemma}

\theoremstyle{definition}
\newtheorem{definition}[theorem]{Definition}

\theoremstyle{remark}
\newtheorem{remark}[theorem]{Remark}

\numberwithin{equation}{section}
\title{The Snapshot Problem for Wave Equations on Homogeneous Trees}

\author{Fulton Gonzalez}
\address{Department of Mathematics,
Tufts University, 
Medford, MA 02155, USA}
\email{fulton.gonzalez@tufts.edu}

\author{Katie Hallett}
\address{Department of Mathematics, St. Lawrence University,
Canton, NY 13617}
\email{khall22@stlawu.edu}

\author{Adelaide Nebeker}
\address{Department of Mathematics, North Carolina State University at Raleigh, Raleigh, NC 27695}
\email{anebeke@ncsu.edu}

\author{Andrew Sailstad}
\address{School of Mathematics, University of Minnesota, Minneapolis, MN 55455}
\email{sails008@umn.edu}

\date{\today}

\begin{document}

\maketitle

\begin{abstract}
By definition, a \emph{wave} on a homogeneous tree $\mathfrak X$ is a solution to the discrete wave equation on $\mathfrak{X}$; that is, a family $\{f_k\}_{k\in\mathbb Z}$ of complex-valued functions on $\mathfrak X$ satisfying the partial difference equation $\mu_1 f_k=(f_{k+1}+f_{k-1})/2$ for all $k$, where $\mu_1$ is the mean value operator on $\mathfrak X$ of radius $1$.  The function $f_k$ is called the \emph{snapshot} of the wave at time $k$.  For $k\geq 2$, we will show that there exist infinitely many waves having given snapshots at times $0$ and $k$, but that all such waves have the same snapshots at times which are multiples of $k$.  For integers $0<k<\ell$, we then consider necessary and sufficient conditions for the existence and uniqueness of a wave with given snapshots at times $0,\,k,\,\ell$. 
\end{abstract}


\section{Introduction}
Fix an integer $q\geq 1$.  A \emph{homogeneous tree $\mathfrak{X}$ of degree $q+1$} is a tree (that is, a connected undirected graph without cycles) in which every vertex has exactly $q+1$ edges.  Such a tree clearly has infinitely many vertices and edges, and its homogeneity allows us to study many structures on the tree from the point of view of harmonic analysis and group theory.  Indeed, there have been numerous recent works on the subject, for example \cite{FTN1991,GPT2003,PT2021}.  These papers in particular highlight the interesting similarities between homogeneous trees on the one hand and Euclidean spaces as well as rank one noncompact Riemannian symmetric spaces on the other, from the  point of view of geometric and harmonic analysis.

   Let $V$ be the set of all vertices of $\mathfrak{X}$. In this paper, we will only consider functions on the vertices of $\mathfrak{X}$, so  for convenience we refer to complex-valued functions on $V$   as \textit{functions on }$\mathfrak{X}$. The vector space of all such functions $f$ on $\mathfrak{X}$ will be denoted by $\mathcal{F}\mathfrak{(X)}$. 
    
    Since $\mathfrak{X}$ is a tree, the shortest (undirected) path between any two vertices $v$ and $w$ in $\mathfrak{X}$ is unique. The \emph{distance} between $v$ and $w$, denoted by $d(v,w)$, is the number of edges in this shortest path. For a given vertex $v$, the \textit{sphere}  of radius $k$ centered at $v$ is the set of all vertices $w$ with $d(w,v)=k$, and will be denoted by $S_k(v)$.  The set of all vertices $u$ with $d(u,v)\leq k$ will be called the \textit{ball} of radius $k$ centered at $v$, and denoted by $B_k(v)$. It is easy to see that $S_k(v)$ contains $(q+1)q^{k-1}$ vertices and $B_k(v)$ contains $(q^{k+1}+q^{k-1}-1)/(q-1)$ vertices. With these notations, we define the radius $k$ mean value operator, whose study is a focus of this paper.

  \begin{definition}\label{def:MVO} For any  integer $k\geq 0$, the radius $k$ \emph{spherical mean value operator $\mu_k:\mathcal{F}\mathfrak{(X)}\to \mathcal{F}\mathfrak{(X)}$} is defined by 

\begin{equation}\label{eq:MVO}
\mu_kf(v)=\frac{1}{(q+1)q^{k-1}}\sum_{w\in S_k(v)}f(w), \qquad f\in\mathcal F(X),\;v\in\mathfrak X.
\end{equation}

\end{definition}

The spherical mean value operator $\mu_k$ is the analogue of the fixed radius mean value operator $M_rf$ on Euclidean spaces (and indeed Riemannian manifolds), for which there is a vast literature.  (See F. 
 John's book \cite{JohnPlaneWaves} as a starting reference.)

The \emph{Laplacian} of any function $f\in\mathcal F(\mathfrak X)$ is the function $\Delta f$ on $\mathfrak X$ given by
\begin{equation}\label{E:laplacian-def}
    \Delta f=\mu_1 f-f.
\end{equation}

The wave equation on $\mathfrak X$ is the discrete analogue of the wave equation on $\mathbb R^n$.
To introduce it, we will assume discrete time dependence for our functions on $\mathfrak{X}$ by considering functions $f:V \times \mathbb{Z} \to \mathbb{C}$. For convenience, for any fixed $k\in\mathbb Z$, we denote the function $f(k,\cdot)$ on $\mathfrak X$ by $f_k$, and we call $f_k$ \emph{the snapshot of} $f$ \textit{at time} $k$. The second partial derivative of $f$ with respect to discrete time is then given by the second symmetric difference 
\begin{equation}\label{E:second-diff}
      \partial^2 _k f \colon = \frac{f_{k+1} + f_{k-1}}{2} - f_k.
\end{equation}
The \emph{wave equation} on $\mathfrak X$ is the second order partial difference equation on functions $f$ on $\mathfrak X\times\mathbb Z$ given by
\begin{equation}\label{E:wave-eqn}
    \Delta f_k = \partial^2_k f.
\end{equation}
Any function $f$ on $\mathfrak X\times\mathbb Z$ which satisfies the wave equation will be called a \emph{wave} on $\mathfrak X$.  It is clear from \eqref{E:laplacian-def} and \eqref{E:second-diff} that the wave equation is equivalent to the equation
    \begin{equation}\label{E:waveeq}
    \mu_1 f_k = \frac{f_{k+1} + f_{k-1}}2,
    \end{equation}
or to  
\begin{equation}\label{eq:waveq2}
    f_{k+1}=2\,\mu_1 f_k - f_{k-1}.
\end{equation}

    It can be seen from the above that if any consecutive snapshots $f_m$ and $f_{m+1}$ are known, a wave $u$ can be determined recursively at all times $k\in \mathbb{Z}$. For wave equations, we will often consider initial conditions on our wave in terms of known consecutive snapshots, particularly at times $k=0$ and $k=1$, which we refer to as \textit{Cauchy data}.  

    It is also clear by induction from \eqref{eq:waveq2} that waves on $\mathfrak X$ propagate with unit speed.  That is to say, if $f$ is a wave on $\mathfrak X$, then for each integer $k$ and each vertex $v$, $f_k(v)$ depends only the values of $f_0$ and $f_1$ at vertices $w$ whose distance from $v$ is at most $|k|$.  However, it is easy to produce examples of waves on $\mathfrak X$ where Huygens' principle is shown not to hold.

    Wave equations were considered by Cohen and Pagliacci in 1994 \cite{CohenPagliacci94}, who (by using Radon transforms among other methods) obtained closed form solutions in terms of the Cauchy data.  Continuous time wave equations were considered by Medolla and Setti in 1999 \cite{MedollaSetti99}, who obtained solutions in terms of convolutions with suitable wave kernels.  (Unfortunately, solutions to the latter type of wave equation are not discretely ``hyperbolic'' and do not propagate at finite speed.)

    Let us now introduce the two- and three-snapshot problems for the wave equation.  Fix an integer $k\geq 2$, and suppose that $\{f_j\}$ is a wave on $\mathfrak X$.  Instead of the Cauchy data $f_0$ and $f_1$,  suppose that we are given the snapshots $f_0$ and $f_k$.  How much of the wave $\{f_j\}$ can we recover?  Moreover, if $f$ and $g$ are arbitrary functions on $\mathfrak X$, is there a wave $\{f_j\}$ such that $f_0=f$ and $f_k=g$?

    Next fix integers $k$ and $\ell$ with $2\leq k<\ell$.  Under what conditions on $k$ and $\ell$ is a wave $\{f_j\}$ uniquely determined by the snapshots $f_0,\;f_k$, and $f_\ell$?  And given three functions $f,\;g$, and $h$ on $\mathfrak X$, is there a wave $\{f_j\}$ such that $f_0=f$, $f_k=g$, and $f_\ell=h$?  Or do we need to impose some consistency conditions on $f,\,g$, and $h$ for such a wave to exist?

    The present paper, which was motivated by the recent paper \cite{CGKW2024}, provides an answer to the above questions.

The material in the paper is organized as follows. In Section \hyperref[sec:MVO]{2} we introduce an formula for the iterated mean value $\mu_k\,\mu_\ell$,  and prove that any polynomial in $\mu_1$ is a surjective linear operator on $\mathcal F(\mathfrak X)$. In Section \hyperref[sec:GenSol]{3}, we provide a closed-form solution to the wave equation on $\mathfrak{X}$ in terms of the Cauchy data $f_0$ and $f_1$ using Chebyshev polynomials, reformulating the earlier general solution offered by Cohen and Pagliacci \cite{CohenPagliacci94} in a way that is useful for the study of snapshot problems. Section \hyperref[sec:2Snap]{4} considers the uniqueness and existence question for a wave having  snapshots at two given times, which we can take to be $t=0,\;k$. Uniqueness is shown to be determined only at snapshots with times a multiple of $k$, and existence is obtained as straightforward consequence of our closed form solution to the wave equation. Section \hyperref[sec:3Snap]{5} explores the three-snapshot problem, with three distinct snapshots at times taken to be $t=0.\;k,\;\ell$. Here the uniqueness of the wave is found only when $k$ and $\ell$ are relatively prime. As for existence, it is shown that three functions on $\mathfrak{X}$ can be specific snapshots of a wave only when they satisfy certain consistency conditions. Whether or not these conditions are sufficient for a wave with the desired snapshots to exist leads to two conditions, one for relatively prime $m$ and $\ell$ and a stronger condition for other cases. Section \hyperref[sec:EPD]{6} provides a characterization of the range of the mean value operator $\mu_k$ in terms of the discrete Euler-Poisson-Darboux (EPD) equation, and provides an expression for $\mu_k$ in terms of Dickson polynomials of $\mu_1$ that is consistent with that obtained by Picardello and Cohen in \cite{MPicardelloCohen1988}.  Section \hyperref[sec:2circle]{7} completes the work of Picardello and Cohen on the Two-Circle Problem by asking the corresponding existence question. Picardello and Cohen provide conditions under which uniqueness is achieved, and these conditions are found to be equivalent to those required for existence.


\section{Basic Properties of the mean value operator}\label{sec:MVO}

The mean value operator $\mu_1$ features prominently in the wave equation, and indeed in all of harmonic analysis, on $\mathfrak{X}$. Moreover, one can quickly deduce from \eqref{E:waveeq} that any closed form solution to the wave equation must involve iterated application of $\mu_1$ to functions on $\mathfrak{X}$. Thus it will be important to consider polynomials in $\mu_1$; that is, linear combinations of powers of $\mu_1$. For a start, it will be useful for us to introduce the following iterated mean value formula on $\mathfrak{X}$, which is the discrete analogue of the iterated mean value formula for $\mathbb R^n$ given by John's formula 4.9c in \cite{JohnPlaneWaves}.  (See also Ch.~VI, Sec.~3  in \cite{Helgason2011} for the formula for the unit sphere $S^{n-1}$, and the generalization in \cite{Rouviere2013} for rank one symmetric spaces.)

\begin{proposition}\label{prop:IMV}
Fix integers $k, \ell \geq 0$, and let $m=\min(k,\ell)$. Then 
\begin{equation}\label{eq:IMV}
    \mu_k\mu_\ell  = 
\frac{q}{q+1}\mu_{k+\ell} +\sum_{i=1}^{m-1} \left(\frac{q-1}{q^{i}(q+1)}\mu_{k+\ell-2i}\right)+\frac{\mu_{|k-\ell|}}{q^{m-1}(q+1)}.
\end{equation}

\end{proposition}

The proof is a straightforward combinatorial argument, involving a count of the number of vertices at distance $\ell$ from each vertex in the sphere $S_k(v)$.  The key here is to note that $\bigcup_{w\in S_k(v)} S_\ell(w)=\bigcup_{0\leq i\leq \ell} S_{k+\ell-2i}(v)$, and that if $s\in S_{k+\ell-2i} (v)$ and $1\leq i\leq m-1$, then $s\in S_\ell(w)$ for exactly $(q-1)q^{i-1}$ vertices $w\in S_\ell(v)$.

It follows immediately from Proposition \ref{prop:IMV} that $\mu_k\mu_\ell=\mu_\ell\mu_k$. (This can also be seen by noting that $\mu_k$ is a convolution operator with a radial kernel; however we will not need to use convolutions in this paper.) If we put $\ell=1$ in \eqref{eq:IMV} then we obtain
    \begin{equation}\label{eq:mu1muk}
        \mu_1\mu_k=\frac{q}{q+1}\mu_{k+1}+\frac{1}{q+1}\mu_{k-1}.
    \end{equation}

Fix an integer $k\geq 0$, Then from \eqref{eq:mu1muk}  one can easily infer by induction on $k$ that $$\mu_k = \sum_{i=0}^{k}r_i\mu_1^i,$$
for some rational constants $r_i$ with $\sum r_i =1$ and $r_k\neq 0$.  (The rational numbers numbers $r_i$ will be explicitly given in Theorem \ref{thm:dicksonmuk}.) Likewise, another induction also using \eqref{eq:mu1muk} shows that
\begin{equation}\label{eq:mu1k}
\mu_1^k=\sum_{i=0}^k t_i\,\mu_i,
\end{equation}
where each $t_i$ is nonnegative and rational, and $\sum_{i=0}^k t_i=1$.
(See \cite{MPicardelloCohen1988} and Sec.~\ref{sec:EPD} below.)  With this relation, we can prove that the mean value operator $\mu_k$ is surjective on $\mathcal F(\mathfrak X)$.  In fact, any nonzero polynomial in $\mu_1$ is a surjective linear operator of $\mathcal F(\mathfrak X)$.  This will be essential in our later characterization of the two and three snapshot problems.

\begin{theorem}\label{thm:muonto}
Let $k$ be a nonnegative integer and consider any linear operator on $\mathcal F(\mathfrak X)$ of the form  $Q=\sum_{i=0}^{k}c_i\mu_i$, where $c_i\in \mathbb{C}$ and $c_k\neq 0$. Then $Q$ is surjective.
\end{theorem}

\begin{proof}
    The theorem being trivial for $k=0$, we can assume that $k$ is positive.
    
     Let $g\in \mathcal F(\mathfrak{X})$. We will construct a function $f$ on $\mathfrak X$ such that $Qf=g$, starting from a reference vertex $o$, by defining $f(v)$ inductively on $d(v,o)=m$. The base case is $d(o,v)=0$, so that $v=o$. Since $c_k\neq 0$ we can choose values of $f$ such that $\mu_kf(o)=g(o)/c_k$ and $\mu_if(o)=0$ for $1\leq i \leq k-1$. (In fact we can put $f(v)=g(o)/c_k$ for all vertices $v$ at distance $k$ from $o$.) Then clearly $Qf(o)=\sum_{i=0}^{k}c_i\mu_if(o)=g(o)$, and for all vertices $v\in B_k(o)$, the ball of radius $k$ centered at $o$, the value of $f(v)$ has been defined.
    
    Now suppose that we have constructed $f$ to satisfy $Qf(v)=g(v)$ for all $v$ with $d(v,o)\leq m$. Then  $f(w)$ has been defined for all  $w\in B_{k+m}(o)$. It follows that for all vertices $v$ with $d(v,o)=m+1,$ the value of $Q'f(v)=\sum_{j=0}^{k-1}c_j\mu_jf(v)$ has been determined, so that $Q'f(v)=c$ for some $c\in \mathbb{C}$. Now it is clear that for each such $v$ there are $q^k$ vertices $s$ such that $d(v,s)=k$ and $d(o,s)=k+m+1$. Since $s\notin B_{k+m}(o)$, we are free to choose the values $f(s)$ such that $\mu_kf(v)=(g(v)-c)/c_k$. It follows that $Qf(v)=c_k\mu_kf(v)+Q'f(v)=g(v)$ as desired. We have thus constructed $f$ to satisfy $f(v)=g(v)$ for all $v$ with $d(v,o)\leq m+1$, finishing the induction.

\end{proof}
\begin{remark}\label{rmk:infmany}
The proof above shows that for any function $g\in\mathcal F(\mathfrak X)$, there are infinitely many functions $f\in\mathcal F(\mathfrak X)$ such that $Qf=g$.  In fact the kernel of $Q$ is infinite-dimensional.
\end{remark}
\begin{remark}\label{rmk:rn-surjectivity}  In $\mathbb R^n$ the problem of surjectivity of convolutions with compactly supported distributions is a much harder problem, which was solved by Ehrenpreis in \cite{Ehrenpreis1960}.
\end{remark}


\section{Closed form solutions to the wave equation}\label{sec:GenSol}

We recall now that given Cauchy data $f_0$ and $f_1$, the wave equation \ref{E:waveeq} offers a recursive method of determining the wave $\{f_k\}$. It is of course desirable to obtain an explicit expression for $f_k$ in terms of  $f_0$ and $f_1$. 
Such an expression was found by Pagliacci in \cite{Pagliacci93}, who provided  an expression of $f_k$ as a linear combination of spherical mean values of the Cauchy data, and whose coefficients were determined through recursion. Below we present an expression that uses the Chebyshev polynomials of the second kind, denoted by $U_n(x)$ (\cite{ChebyshevPolys}).  These are defined by $U_n(\cos\theta)=(\sin(n+1)\theta)/\sin\theta$, or recursively by $U_0(x)=1$, $U_1(x)=2x$, and $U_k(x)=2xU_{k-1}(x)-U_{k-2}(x)$ for $k\geq 2$. The Chebyshev polynomials with negative indices are defined by $U_{-k-1}=-U_{k-1}$ for $k\geq 0$. (Note that this implies that $U_{-1}(x)=0$ and that the recursion formula holds for all indices $k$, regardless of sign.)  Let us now consider the linear operators $U_n(\mu_1)$ on $\mathcal F(\mathfrak X)$, which allow us to solve the Cauchy problem for the wave equation on $\mathfrak X$.

\begin{theorem}\label{thm:genform}
   The wave equation \eqref{E:wave-eqn} with Cauchy data $f_0$ and $f_1$ on $\mathfrak X$ has the unique solution
    \begin{equation}\label{eq:closedform soln}
    f_k = U_{k-1}(\mu_1)f_1-U_{k-2}(\mu_1)f_0, \qquad\qquad k\in\mathbb Z.
    \end{equation}
\end{theorem}

\begin{proof}
The recursion formula for $U_k$ easily shows that if $f_k$ is defined by \eqref{eq:closedform soln} for $k\neq 0,\,1$, then the two-sided sequence $\{f_k\}$ solves the wave equation \eqref{E:waveeq} with Cauchy data $f_0$ and $f_1$.

For the uniqueness, it has already been observed after \eqref{eq:waveq2} that any two consecutive snapshots of a wave determine the wave uniquely. 
\end{proof}

The closed form expression \eqref{eq:closedform soln} is likely well known, although we did not find the formula in Pagliacci's papers \cite{Pagliacci93,CohenPagliacci94}, nor in a review of the literature. (We did discover an expression for the solution of a related but slightly different wave equation in \cite{peterson2025discretewaveequationapplications} which likewise uses Chebyshev polynomials.)  The formula \eqref{eq:closedform soln} will prove to be useful in the sequel, especially as we shall be using certain classical properties of Chebyshev polynomials. 

We will generalize Theorem \ref{thm:genform} in Theorem \ref{thm:gensolEPD} below. 

\section{The Two Snapshot Problem}\label{sec:2Snap}

Our objective in this section is to show that if $k\neq 0$ and if we are given snapshots $f_0$ and $f_k$ of a wave $\{f_j\}$ on $\mathfrak{X}$ we can uniquely determine the snapshots $f_{nk}$ for all integers $n$, and that, in addition, if $k\neq \pm 1$, then there are infinitely many such waves $\{f_j\}$.  To this end, we will first generalize the wave equation \eqref{E:wave-eqn} by obtaining a relation between $f_k$,\; $f_{k+m}$, and $f_{k-m}$. This will then give us a formula for $f_{nk}$ in terms of $f_0$ and $f_k$.

For this, we now recall the Chebyshev polynomials $T_m$ of the first kind.  These are defined by $T_m(\cos\theta)=\cos m\theta$ for $m\in\mathbb Z$, or recursively by $T_0(x)=1,\;T_1(x)=x$, and $T_{m+1}(x)=2x\,T_m(x)-T_{m-1}(x)$. Putting $T_{-m}(x)=T_m(x)$, we now note that these recursion relations then hold for all integers $m$.

\begin{lemma}\label{thm:f+f=T}  Let $\{f_j\}$ be a wave on $\mathfrak{X}$.  Then for all integers $m$ and $k$ we have
        \begin{equation}\label{eq:recurs1}
        f_{m+k} + f_{m-k}  = 2T_k(\mu_1)f_m,\qquad k\in\mathbb Z.
        \end{equation}
\end{lemma}
\begin{proof}
It is enough to prove \eqref{eq:recurs1} by induction on $k\geq 0$ (and for all integers $m$), noting that it is trivial when $k=0$ and is the wave equation \eqref{E:waveeq} when $k=1$. 

Now let $k\geq 1$ and assume that \eqref{eq:recurs1} holds for this $k$. 
 Then \eqref{eq:waveq2} shows that
\begin{align*}
    f_{m+(k+1)} = 2\mu_1f_{m+k} - f_{m+(k-1)}\\
    f_{m-(k+1)} = 2\mu_1f_{m-k} - f_{m-(k-1)}\\
\end{align*}
for all integers $m$.  Adding the two equations and applying the induction hypothesis yields 
\begin{align*}
f_{m+(k+1)}+f_{m-(k+1)}&= 2[2\mu_1T_k(\mu_1)f_m-T_{k-1}(\mu_1)f_m]\\
&=2\,T_{k+1}(\mu_1)\,f_m,
\end{align*}
completing the induction.
\end{proof}

\begin{theorem}\label{thm:constmultform}
    Let $\{f_j\}$ be a wave on $\mathfrak X$. Then for all integers $n$ and $k$ we have
    \begin{equation}\label{eq:fnk}
    f_{nk} = U_{n-1}(T_k(\mu_1))f_k - U_{n-2}(T_k(\mu_1))f_0.
    \end{equation}
\end{theorem}
\begin{proof} We first observe that \eqref{eq:fnk} can be easily verified for $n=0$ or $k=0$, the relation for $n=0$ being trivial, and the relation for $k=0$ using the fact that $U_n(1)=n+1$.  For the rest of the proof we can assume that $n\neq 0$ and $k\neq 0$.
    
    We now prove \eqref{eq:fnk} for $n\geq 1$ and for all $k\neq 0$ by induction on $n$. Note that when $n=1$, \eqref{eq:fnk} is trivial and when $n=2$, \eqref{eq:fnk} reduces to \eqref{eq:recurs1} with $m=k$.  

    Next let $n\geq 2$ and assume that the relation \eqref{eq:fnk} holds for $f_{mk}$ with $1\leq m\leq n$.  Then by \eqref{eq:recurs1} with $m=nk$ and the recursion relations for $U_n$, we obtain
    \begin{align*}
        f_{(n+1)k}&=2\,T_k(\mu_1)\,f_{nk} - f_{(n-1)k}\\
        &=2\,T_k(\mu_1)\left(U_{n-1}(T_k(\mu_1))\,f_k-U_{n-2}(T_k(\mu_1))\,f_0\right)\\
        &\qquad\qquad\qquad -U_{n-2}(T_k(\mu_1))\,f_k+U_{n-3}(T_k(\mu_1))\,f_0\\
        &=U_n(T_k(\mu_1))\,f_k-U_{n-1}(T_k(\mu_1))\,f_0,
    \end{align*}
    completing the induction.

    For $n$ negative, we replace $n$ by $-n$ (with $n$ positive), and apply the already-proven part of the relation \eqref{eq:fnk} to the wave $\{g_j\}$, where $g_j=f_{k-j}$.
\end{proof}

\begin{remark}
While Theorem~\ref{thm:constmultform} gives an explicit expression for $f_{nk}$ in terms of $f_0$ and $f_k$, we observe that Lemma~\ref{thm:f+f=T} already allows us to recursively determine a wave at snapshots which are multiples of $k$. First, we note that $f_{-k}$ may be determined in terms of the given snapshots $f_k$ and $f_0$, since $f_{-k} = 2T_k(
\mu_1)f_0 - f_k$. Then, we may determine $f_{k+k}=f_{2k}$ and $f_{-k+-k}=f_{-2k}$ as expressions of now-known terms. Through iteration, we may obtain $f_{\pm 3k}$, $f_{\pm4k}$, and so forth, on to any desired multiple of $k$.  
\end{remark}

With that said, both formulations for obtaining $f_{nk}$  will be useful as we consider the three snapshot problem in Sec.~\ref{sec:3Snap}.   Thus the snapshots $f_{nk}$ are uniquely determined once we are given $f_0$ and $f_k$.

The following result addresses the question of whether a wave exists with two given snapshots at two different times.

\begin{theorem}\label{thm:infsnapshots} Fix an integer $k\geq 2$.  For arbitrary functions $g$ and $h$ on $\mathfrak{X}$, there exist infinitely many waves $\{f_j\}$ on $\mathfrak X$  such that $f_0=g$ and $f_k=h$.
\end{theorem}
\begin{proof}
For any wave $\{f_j\}$ on $\mathfrak X$ such that $f_0=g$ and $f_k=h$, the closed form solution \eqref{eq:closedform soln} shows that
\begin{equation}\label{eq:altclosedform}
U_{k-1}(\mu_1)\,f_1 = h-U_{k-2}(\mu_1)\,g.
\end{equation}
Since the Cauchy data $\{f_0,\,f_1\}$ determine the wave $\{f_j\}$, it suffices to prove that there are infinitely many functions $f_1$ on $\mathfrak X$ satisfying the relation \eqref{eq:altclosedform}.

Now \eqref{eq:mu1k} shows that $U_{k-1}(\mu_1)=\sum_{i=1}^{k-1} c_i\,\mu_i$ for some rational numbers $c_i$, with $c_{k-1}\neq 0$.   By Theorem \ref{thm:muonto} and the subsequent Remark \ref{rmk:infmany}, there exist infinitely many functions $f_1\in\mathcal F(\mathfrak X)$ satisfying \eqref{eq:altclosedform}.  This proves the theorem.
\end{proof}

To summarize, we have shown that for $k>1$ there exist infinitely many waves having given snapshots at times $0$ and $k$, but that all such waves have the same snapshots at times which are multiples of $k$.

\section{The Three Snapshot Problem}\label{sec:3Snap}

We now turn to the problem of determining a wave $\{f_j\}$ if we are given the snapshots $f_0$, $f_k$ and $f_\ell$, where $2\leq k < \ell$. For such a wave, one expects to use the results of the preceding section to be able to determine the wave at all times that are integer linear combinations of $k$ and $\ell$. 

\begin{lemma}\label{thm:gcdunique}
Suppose that we are given the snapshots $f_0,\,f_k$, and $f_\ell$ of a wave $\{f_j\}$ on $\mathfrak{X}$. Then for any integers $r$ and $s$, $f_{rk+s\ell}$ is is given by the formula
\begin{equation*}
    f_{rk+s\ell}=T_{s\ell}(\mu_1)f_{rk}+T_{rk}(\mu_1)f_{s\ell}-T_{rk-s\ell}(\mu_1)f_0.
\end{equation*}
\end{lemma}

\begin{proof}
   By Theorem \ref{thm:f+f=T} we have
\begin{equation*}
    f_{rk+s\ell}+f_{rk-s\ell}=2T_{s\ell}(\mu_1)f_{rk},
\end{equation*}
and 
\begin{equation*}
    f_{s\ell+rk}+f_{s\ell-rk}=2T_{rk}(\mu_1)f_{s\ell}.
\end{equation*}
Adding the two equations and applying Theorem \ref{thm:f+f=T} yet again yields the desired result. 
\end{proof}
This result immediately implies the following theorem.
\begin{theorem}\label{thm: 3snapuniqueness}
 Let $d=\gcd(k,\ell)$. Then $f_d$ is determined by $f_0$, $f_k$, and $f_\ell$.  In particular, if $k$ and $\ell$ are relatively prime, then the wave $\{f_j\}$ is uniquely determined by $f_0,\,f_k$, and $f_\ell$.
\end{theorem}

Since $f_0$ and $f_d$ are determined, we can also determine the snapshots $f_{md}$. Among these snapshots are $f_k$ and $f_\ell$.  In case $d>1$, Theorem~\ref{thm:infsnapshots} states that there are infinitely many waves with given snapshots at times $0$ and $d$. In this case then, if there exists a wave with given snapshots at times $0,\;k$, and $\ell$, then infinitely many such waves exist.

This completes our characterization of the uniqueness of a wave given three snapshots.

We can now turn to the question of existence.  Explicitly, fix positive integers $k$ and $\ell$, with $k<\ell$, and suppose that $f,\,g$, and $h$ are functions on $\mathfrak X$.  Does there exist a wave $\{f_j\}$ on $\mathfrak X$ such that $f_0=f,\,f_k=g$, and $f_\ell=h$?  (We already know that if such a wave exists, it is unique iff $k$ and $\ell$ are relatively prime.)

If such a wave $\{f_j\}$ does exist, then the closed form solution \eqref{eq:closedform soln} implies that 
\begin{equation}\label{eq:4.1}
    g=U_{k-1}(\mu_1)f_1-U_{k-2}(\mu_1)f
\end{equation}
and 
\begin{equation}\label{eq:4.2}
    h=U_{\ell-1}(\mu_1)f_1-U_{\ell-2}(\mu_1)f.
\end{equation}
Eliminating $f_1$, we see that these equations imply that $f,\,g$ and $h$ must satisfy the  \textit{compatibility condition} 
\begin{equation}\label{eq:4.3}
    U_{\ell-1}(\mu_1)(g+U_{k-2}(\mu_1)f) = U_{k-1}(\mu_1)(h+U_{\ell-2}(\mu_1)f).
\end{equation}
 We call any ordered triple of functions $(f,g,h)$ on $\mathfrak{X}$ satisfying \ref{eq:4.3} a \emph{compatible triple}. 

Thus a necessary condition for the existence of a wave with the snapshots $f,\,g$, and $h$ at the three fixed times is that they be a compatible triple.   A natural question to ask is whether the compatibility condition \eqref{eq:4.3} is also sufficient.   That is to say, suppose that  $(f,\,g,\,h)$ is a compatible triple.  Is there a wave $\{f_j\}$ such that $f_0=f,\,f_k=g$, and $f_\ell = h$?

It is not hard to see that this boils down to the question of whether the compatibility relation \eqref{eq:4.3} implies the existence of a function $f_1$ on $\mathfrak X$ for which \eqref{eq:4.1} and \eqref{eq:4.2} hold. One may observe that this in turn reduces to a problem in linear algebra,  for which the following result will be useful.

\begin{proposition}\label{prop:linalg}
      Let $V$ be a vector space (over some field) and suppose that $S$ and $T$ are surjective linear operators on $V$ such that $ST=TS$.  Then the following conditions are equivalent:
    \begin{enumerate}[(i)]
        \item For any vectors $v$ and $w$ in $V$ such that $Sv=Tw$, there exists a vector $u\in V$ such that $Su=w$ and $Tu=v$.        
        \item $S:\ker{T} \to \ker{T}$ is surjective.    
    \end{enumerate}

\begin{proof}
Assume (i) holds. For any $w\in\ker{T}$, we have $S(0)=0=Tw$, so there exists $u\in V$ such that $Tu=0,\;Su=w$, proving (ii). 

Conversely, suppose (ii) holds, and $v$ and $w$ are vectors in $V$ such that $Sv=Tw$.  Since $T$ is surjective, there is a $v_1\in V$ with $Tv_1=v$.  Since $S$ and $T$ commute, this implies that $T(w-Sv_1)=0$.  By (ii), there is a vector $v_2\in\ker T$ such that $w-Sv_1=Sv_2$.  Then $u=v_1+v_2$ satisfies $Su=w$ and $Tu=Tv_1=v$.
\end{proof}

\end{proposition}

This simple result leads directly to conditions under which any compatible triple is the set of snapshots of a wave at times $0,\,k,\,\ell$. Note that by Theorem \ref{thm:muonto}, the linear operators $U_{k-1}(\mu_1)$ and $U_{\ell-1}(\mu_1)$ are surjective and commuting linear operators on $\mathcal{F}\mathfrak{(X)}$.
\begin{theorem}\label{thm:waveiffsurj}  Fix positive integers $k$ and $\ell$, with $k<\ell$.  Then the following conditions are equivalent:
   
    \begin{enumerate}[(i)]
        \item For any compatible triple $(f,\,g,\,h)$ of functions on $\mathfrak X$, there a exists a wave with snapshots $f_0=f$, $f_k=g$, $f_\ell=h$.
        \item For any compatible triple $(f,\,g,\,h)$, there exists an $f_1\in \mathcal{F}\mathfrak{(X)}$ such that $U_{k-1}(\mu_1)f_1=g+U_{k-2}(\mu_1)f$ and $U_{\ell-1}(\mu_1)f_1=h+U_{\ell-2}(\mu_1)f$.        
        \item The map $U_{\ell-1}(\mu_1):\ker{U_{k-1}(\mu_1)} \to \ker{U_{k-1}(\mu_1)}$ is surjective.    
    \end{enumerate}
\end{theorem}

\begin{proof} This is clear from Proposition \ref{prop:linalg}.
\end{proof}

Thus, the existence of a wave is equivalent to the surjectivity of $U_{\ell-1}(\mu_1)$ restricted to $\ker{U_{k-1}(\mu_1)}$. Fortunately, an explicit characterization of $\ker{U_{k-1}(\mu_1)}$ is not needed, but rather surjectivity can be evaluated by considering the polynomial factorizations of $U_{k-1}(\mu_1)$ and $U_{\ell-1}(\mu_1)$, which allows us to determine existence from condition \ref{eq:4.3} for relatively prime $k$ and $\ell$.

\begin{theorem}\label{thm:3snapex}
    
     Let $k$ and $\ell$ be relatively prime. For all $f,g,h\in F(\mathfrak{X})$, the compatibility condition 

    $$U_{k-1}(\mu_1)(g+U_{\ell-2}(\mu_1)f)=U_{\ell-1}(\mu_1)(h+U_{k-2}(\mu_1)f)$$

    is necessary and sufficient for the existence of a unique wave with $f_0=f$, $f_\ell=g$, and $f_k=h$.
\end{theorem}

\begin{proof}
    We note that $U_{k-1}(x)$ and $U_{\ell-1}(x)$ are have no common roots if and only if $k$ and $\ell$ are relatively prime. (See Szego's book \cite{Szego1975}.) Thus, $U_{k-1}(x)$ and $U_{\ell-1}(x)$ are relatively prime if and only if $k$ and $\ell$ are relatively prime. It follows that there exist polynomials $Q(z)$ and $R(z)$ satisfying $U_{k-1}(z)Q(z)+U_{\ell-1}(z)R(z)=1$, so that

    $$U_{k-1}(\mu_1)Q(\mu_1)+U_{\ell-1}(\mu_1)R(\mu_1)=\mu_0,$$ 

    where $\mu_0f=f$ for all $f\in \mathcal{F}(\mathfrak{X})$.

    Now suppose that $f\in \ker{U_{k-1}(\mu_1)}$. It follows that
    
    $$f=\mu_0f =(U_{k-1}(\mu_1)Q(\mu_1)+U_{\ell-1}(\mu_1)R(\mu_1))f,$$

    which reduces to 
    
    $$U_{\ell-1}(\mu_1)R(\mu_1)f=f.$$ 
    
Let $g=R(\mu_1)f$. Then $g\in \ker{U_{k-1}(\mu_1)}$ and $U_{\ell-1}g=f$, so that $U_{\ell-1}(\mu_1):\ker{U_{k-1}(\mu_1)} \to \ker{U_{k-1}(\mu_1)}$ is surjective. It follows by \ref{thm:waveiffsurj} that a wave exists, and by \ref{thm: 3snapuniqueness} this wave must be unique.

\end{proof}

   When $k$ and $\ell$ are not relatively prime, the compatibility condition \ref{eq:4.3} is not sufficient for the existence of a wave. For example, if $k=2$ and $\ell=4$, one can show that $U_{3}(\mu_1):\ker{U_{1}(\mu_1)}\to\ker{U_{1}(\mu_1)}$ is not surjective, 
   as $\ker{U_1(\mu_1}) \subset \ker{U_3(\mu_1)}$. The following theorem provides a stronger compatibility condition that is sufficient for all $k$ and $\ell$.

   \begin{theorem}\label{thm:k-ell-arbitrary}
    
     Fix $k,\ell \in \mathbb{Z}$ with ${2\leq k < \ell}$, and let $W(x)=\dfrac{U_{k-1}(x)}{U_{d-1}(x)}$ and $V(x)=\dfrac{U_{\ell-1}(x)}{U_{d-1}(x)}$, where $d=\gcd(k,\ell)$. Then for all $f,g,h\in \mathcal{F}(\mathfrak{X})$, the condition 

    \begin{equation}\label{eq:strongCC}
        W(\mu_1)(g+U_{\ell-2}(\mu_1)f)=V(\mu_1)(h+U_{k-2}(\mu_1)f)
    \end{equation}

    is necessary and sufficient for the existence of a wave with $f_0=f$, $f_\ell=g$, and $f_k=h$.
    
\end{theorem}

\begin{proof}
     For any $k$, $\ell$, we have $\gcd(U_{k-1}(x),U_{\ell-1}(x))=U_{d-1}(x)$ as shown in \cite{RTW05}. It follows that $V(\mu_1)$ and $W(\mu_1)$ must be relatively prime polynomials, and applying the same argument as in Theorem \ref{thm:3snapex}, we see that $W(\mu_1):\ker{V(\mu_1)} \to \ker{V(\mu_1)}$ is surjective. Thus by Proposition \ref{prop:linalg} there is some function $f^*$ satisfying 
     
     $$V(\mu_1)f^*=g+U_{k-2}(\mu_1)f$$ 
     and 
     $$W(\mu_1)f^*=h+U_{\ell-2}(\mu_1)f.$$ 
     
     Now, by Theorem \ref{thm:muonto} there is a function $f_1$ such that $U_{d-1}(\mu_1)f_1=f^*$; it follows that condition (ii) of Theorem \ref{thm:waveiffsurj} is satisfied by $f_1$, and thus a wave exists.
    \end{proof}

    \begin{remark}
It is clear that Theorem \ref{thm:k-ell-arbitrary} generalizes Theorem \ref{thm:3snapex}, since in the case where $k$ and $\ell$ are relatively prime, $d=1$ so that $U_{d-1}(x)=U_0(x)=1$.
    \end{remark}

Now, again fixing $2\leq k < \ell$, an ordered triple of functions $(f,g,h)\in\mathcal{F}(\mathfrak{X})$ will be called a \emph{snapshot triple} if there exists a wave on $\mathfrak{X}$ whose snapshots at times $0, k$ and $\ell$ are $f$, $g$ and $h$ respectively.

With this definition, let us now clarify the relation between compatible triples and snapshot triples vis-\`{a}-vis the range $U_{\ell -1}(\mu_1)(\ker U_{k-1}(\mu_1)$).

Fix a wave $\{g_k\}_{k\in\mathbb{Z}}$ on $\mathfrak{X}$ such that $g_0 = f$, $g_k = g$ for some functions $f$, $g$ $\in \mathcal{F}(\mathfrak{X})$. Then by \ref{thm:constmultform}, there exist infinitely many such waves. In any case, $(f,g,g_\ell)$ forms a snapshot triple. By Theorem \ref{thm:3snapex}, the snapshot $g_\ell$ of our wave is given by 
    $$ g_\ell = U_{\ell-1}(\mu_1)g_1 -U_{\ell-2}(\mu_1)g_0.$$
    Now, suppose that $h\in\mathcal{F}(\mathfrak{X})$ such that $(f,g,h)$ is a compatible triple. Then
        $$U_{\ell-1}(\mu_1)(g + U_{k-2}(\mu_1)f) = U_{k-1}(\mu_1)(h + U_{\ell-2}(\mu_1)f)$$

    But we also have 
    
            $$
            U_{\ell-1}(\mu_1)(g + U_{k-2}(\mu_1)f)
            =U_{k-1}(\mu_1)(g_\ell + U_{\ell-2}(\mu_1)f)
            $$

    This implies that $h\in g_\ell + \ker U_{k-1}(\mu_1)$. Conversely, it is easy to see that if $h$ belongs to $g_\ell + \ker U_{k-1}(\mu_1)$, then $(f,g,h)$ is a compatible triple.

    Now suppose that $h$ is a function on $\mathfrak{X}$ such that $(f,g,h)$ is a snapshot triple. Then, there is a wave $\{h_k\}_{k\in\mathbb{Z}}$ such that $h_0 = f$ and $h_k=g$, and $h_\ell = h$. Then we have

    \begin{align*}
        h_\ell = U_{\ell-1}(\mu_1)h_1  - U_{\ell-2}(\mu_1)f\\
        g_\ell = U_{\ell-1}(\mu_1)g_1  - U_{\ell-2}(\mu_1)f\\
    \end{align*}

    which implies that $h - g_\ell = U_{\ell-1}(\mu_1)(h_1-g_1)$. Now, since $g = g_k = h_k$ we have 
    \begin{align*}
       g = U_{k-1}(\mu_1)g_1 - U_{k-2}(\mu_1)f\\
       g = U_{k-1}(\mu_1)h_1 - U_{k-2}(\mu_1)f\\
    \end{align*}

    which implies that $h_1-g_1 \in \ker U_{k-1}(\mu_1)$. This shows that $h_\ell \in g_\ell + U_{\ell-1}(\mu_1)(\ker U_{k-1}(\mu_1))$. Conversely, one can also show that $(g_0, g_k, h)$ is a snapshot triple whenever $h\in g_\ell + U_{\ell-1}(\mu_1)(\ker U_{k-1}(\mu_1))$.

    From this we conclude that a compatible triple $(f,g,h)$ is a snapshot triple if and only if $U_{\ell-1}(\mu_1)$ is surjective on $\ker U_{k-1}(\mu_1)$, a fact which we have already verified in Theorem \ref{thm:waveiffsurj}(iii).


\section{Mean Value Operators and the Euler-Poisson-Darboux equation on $\mathfrak{X}$}\label{sec:EPD}

In classical harmonic analysis, the \emph{Euler-Poisson-Darboux (EPD) equation}  plays a central role, and in particular characterizes the range of the mean value operator on $\mathbb{R}^n$. For homogeneous trees $\mathfrak{X}$, we define an analogue of this operator. 

\begin{definition}
    Let $\{f_k\}_{k=1}^\infty$ be a sequence of complex-valued functions on $\mathfrak{X}$. We say that $\{f_k\}$ is a \emph{solution to the Euler-Poisson-Darboux (EPD) equation} provided that
    \begin{equation}\label{eq:EPD}
    \mu_1 f_k = \frac{q}{q+1}f_{k+1} + \frac{1}{q+1}f_{k-1}, \qquad k\geq 1
    \end{equation}
\end{definition}

We will solve this equation in general for given initial data $f_0$ and $f_1$, but we note that the following result is an immediate consequence of the iterated mean value relation \eqref{eq:mu1muk}.

\begin{proposition}\label{prop:evenEPDsol}
    Let $\{f_k\}_{k=1}^\infty$ be a sequence of functions on $\mathfrak{X}$ such that $f_1=\mu_1 f_0$. Then $\{f_k\}$ satisfies the EPD equation iff for all integers $k\geq 1$,  $f_k = \mu_kf_0$.
\end{proposition}

\begin{definition}\label{def:genEPD}
    Let $s$ and $t$ be positive real numbers such that $s+t=1$.  A sequence $\{f_k\}_{k=1}^\infty$ of functions on $\mathfrak X$ is said to satisfy the \emph{generalized EPD equation} provided that
\begin{equation}\label{eq:genEPD}
    \mu_1 f_k=s\,f_{k+1}+t\, f_{k-1},\qquad k\geq 1.
\end{equation}
\end{definition}

Note that \eqref{eq:genEPD} specializes to the wave equation \eqref{E:waveeq} when $s=1/2$ and the EPD equation \eqref{eq:EPD} when $s=q/(q+1)$.

We will now obtain a closed-form solution to the generalized EPD equation for given Cauchy data $f_0$ and $f_1$. This will allow us to derive an explicit formula expressing the mean value operator $\mu_k$ as a $k$-th degree polynomial in $\mu_1$. Both will be expressed in terms of Chebyshev polynomials of the second kind.

\begin{theorem}\label{thm:gensolEPD}

Let $\{f_k\}_{k=1}^\infty$ be a sequence of functions on $\mathfrak{X}$ satisfying the generalized Euler-Poisson-Darboux equation \eqref{eq:genEPD}, with Cauchy data $f_0$ and $f_1$. Then for $k\geq 0$, the snapshot $f_k$ is given by \begin{equation}\label{eq:closedform2}
f_k=\left(\sqrt{\frac ts}\right)^k\,\left[\sqrt{\frac st}\, U_{k-1}\left(\frac{1}{2\sqrt{st}}\mu_1\right) f_1-U_{k-2}\left(\frac{1}{2\sqrt{st}}\mu_1\right) f_0\right]
\end{equation}
\end{theorem}

\begin{remark}
    The formula \eqref{eq:closedform soln} solving the wave equation is a special case of \eqref{eq:closedform2}, where $s=t=1/2$.
\end{remark}

\begin{proof}
We use strong induction on $k$.  Since $U_0(x)=1,\,U_{-1}(x)=0$, and $U_{-2}(x)=-U_0(x)=-1$, it is clear that the equality \eqref{eq:closedform2} holds when $k=0,\,1$.  So let $k\geq 1$ and assume that the equality \eqref{eq:closedform2} holds for all $\ell\leq k$.  Using the recursion formula for $U_\ell$, we obtain
\begin{align*}
    f_{k+1}&=\frac 1s\,\mu_1 f_k-\frac ts\, f_{k-1}\\
    &=\frac 1s\,\left(\sqrt{\frac ts}\right)^k\,\left[\sqrt{\frac st}\,\mu_1\, U_{k-1}\left(\frac{1}{2\sqrt{st}}\mu_1\right) f_1-\mu_1\,U_{k-2}\left(\frac{1}{2\sqrt{st}}\mu_1\right) f_0\right]\\
    &\qquad -\frac ts\,\left(\sqrt{\frac ts}\right)^{k-1}\,\left[\sqrt{\frac st}\, U_{k-2}\left(\frac{1}{2\sqrt{st}}\mu_1\right) f_1-U_{k-3}\left(\frac{1}{2\sqrt{st}}\mu_1\right) f_0\right]\\
    &=\left(\frac ts\right)^{k/2}\,\left[2\,\left(\frac{1}{2\sqrt{st}}\,\mu_1\right)\,U_{k-1}\left(\frac{1}{2\sqrt{st}}\,\mu_1\right)-
    U_{k-2}\left(\frac{1}{2\sqrt{st}}\,\mu_1\right)\right]\,f_1\\
    &\;-\left(\frac ts\right)^{(k+1)/2}\,\left[2\,\left(\frac{1}{2\sqrt{st}}\,\mu_1\right)\,U_{k-2}\left(\frac{1}{2\sqrt{st}}\,\mu_1\right)-
    U_{k-3}\left(\frac{1}{2\sqrt{st}}\,\mu_1\right)\right]\,f_1\\
    &=\left(\frac ts\right)^{k/2}\,U_k\left(\frac{1}{2\sqrt{st}}\,\mu_1\right)\,f_1
    -\;\left(\frac ts\right)^{(k+1)/2}\,U_{k-1}\left(\frac{1}{2\sqrt{st}}\,\mu_1\right)\,f_0\\
    &=\left(\sqrt{\frac ts}\right)^{(k+1)/2}\,\left[\sqrt{\frac st}\, U_{k}\left(\frac{1}{2\sqrt{st}}\mu_1\right) f_1-U_{k-1}\left(\frac{1}{2\sqrt{st}}\mu_1\right) f_0\right],
\end{align*}
completing the induction step.

\end{proof}

Proposition \ref{prop:evenEPDsol} and Theorem \ref{thm:gensolEPD} allow us to obtain a closed form expression of $\mu_k$ as a $k$-th degree polynomial in $\mu_1$. We note that such an expression is not novel. In 1988, Cohen and Picardello \cite{PCPompeiu} obtained:

\begin{proposition}\label{prop:sigmapolys}\cite{PCPompeiu}
    Let the operator $\Sigma_m = (q+1)q^{m-1}\mu_m$ for all integers $m$. 
    Then $\Sigma_1\Sigma_1$ = $\Sigma_2 + (q+1)\Sigma_0$, and for $n>1$, $\Sigma_1\Sigma_n = \Sigma_{n+1}+q\Sigma_{n-1}$.

    Now let $p(x)$ be a polynomial such that for $n>2$, $p_{n+1}(x) = xp_n(x) - qp_{n-1}(x)$, with initial data $p_0(x) = 1$, $p_1(x) = x$, and $p_2(x) = x^2 - (q+1)$. Taken together, it follows that $\Sigma_n = p_n(\Sigma_1)$.

\end{proposition}

They then found a trigonometric form of $p_n(x)$ using Cohen's earlier study of a similar recursive polynomial in \cite{Cohen82} where instead $q=2t-1$. Pagliacci also references this recursive polynomial from \cite{Cohen82} in the proof of his solution to the wave equation \eqref{E:wave-eqn} in \cite{Pagliacci93}. He places their origin in a 1930 paper by Geronimus \cite{GeronimusPolys}\textemdash indeed, this family of Chebyshev-type orthogonal polynomials are sometimes called \textit{Geronimus polynomials} (cf. \cite{Lebedev}, \cite{PEHERSTORFER1992}).

If we put $s=q/(q+1)$ and $t=1/(q+1)$ in \eqref{eq:genEPD}, and let $f_1=\mu_1 f_0$, then the uniqueness of the solution to the Cauchy problem for the EPD equation in Proposition \ref{prop:evenEPDsol} gives us the following result.

\begin{theorem}\label{thm:dicksonmuk}
    The mean value operator $\mu_k$ on $\mathfrak{X}$ as a polynomial in $\mu_1$ is given by the following expression in Chebyshev polynomials of the second kind:

    $$\mu_k = \left(\frac{1}{\sqrt{q}}\right)^k\left[\sqrt q\,\mu_1\,U_{k-1}\left(\frac{q+1}{2\sqrt{q}}\mu_1\right) - U_{k-2}\left(\frac{q+1}{2\sqrt{q}}\mu_1\right)\right]$$
    
\end{theorem}

\section{The Two-Circle Existence Problem for Mean Value Operators on $\mathfrak{X}$}\label{sec:2circle}

The following result due to Cohen and Picardello \cite{PCPompeiu} provides the solution to the Two-Circle Pompeiu Problem on $\mathfrak{X}$.

\begin{theorem}\label{6.1}\cite{PCPompeiu}
    Let $k$ and $\ell$ be distinct integers that are not both odd. Let $q+1$ be the degree of $\mathfrak{X}$, and suppose additionally that if $q=2$ then $k$ and $\ell$ are  not both congruent to $4 \mod 6$. Then for any function $f$ on  $\mathfrak{X}$, $\mu_kf=\mu_{\ell}f=0$ if and only if $f=0$.

\end{theorem}

The Pompeiu problem poses a uniqueness question for functions given two mean values\textemdash from Theorem \ref{6.1}, we can see that for arbitary functions $f,g$, $\mu_kf=\mu_kg$ and $\mu_\ell f=\mu_\ell g$ if and only if $f=g$ when the conditions are satisfied. Analogous to the snapshot problem, we can ask the corresponding existence question: for functions $g,h$ on $\mathfrak{X}$, what conditions are necessary and sufficient to guarantee the existence of a function $f$ such that $\mu_kf=g$ and $\mu_\ell f = h$?

If such an $f$ exists, then the necessary condition $\mu_\ell g = \mu_k h$ quickly follows. It can be shown that this condition is also sufficient. Indeed, the conditions in Theorem \ref{6.1} are equivalent to the conditions for $\Sigma_{k}$ and $\Sigma_{\ell}$ (cf. \ref{prop:sigmapolys}) to be relatively prime polynomials of $\Sigma_1$, and the condition $\mu_\ell g = \mu_k h$ is equivalent to the condition $\Sigma_\ell G = \Sigma_k H$, where $G=q^\ell g$ and $H=q^k h$. Thus, when the conditions on $k$ and $\ell$ in Theorem $6.1$ are satisfied, applying the same reasoning as in Theorem $5.7$ yields the existence of an $f$ satisfying the conditions. This is stated in the following theorem.

\begin{theorem}
    Let $g$ and $h$ be functions on $\mathfrak{X}$, and suppose $k$ and $\ell$ satisfy the conditions in Theorem $\ref{6.1}$. Then $\mu_\ell g=\mu_k h$ if and only if there exists a function $f$ such that $\mu_k f=g$ and $\mu_\ell f=h$.

\end{theorem}

\begin{proof}
    Necessity is quickly apparent, so we show only sufficiency. Suppose that $\mu_\ell g= \mu_k h $. Then $\Sigma_\ell G = \Sigma_k H$, where $G=q^\ell g$ and $H=q^k h$. By Theorem \ref{thm:dicksonmuk} we can write $\Sigma_\ell$ and $\Sigma_k$ as polynomials $P_{\ell}(\Sigma_1)$ and $P_k(\Sigma_1)$. Picardello and Cohen showed in \cite{PCPompeiu} that $P_k$ and $P_\ell$ share no common roots. It follows that there are polynomials $Q(\Sigma_1)$ and $R(\Sigma_1)$ satisfying 

    $$\Sigma_k Q(\Sigma_1)+\Sigma_\ell R(\Sigma_1)=\Sigma_0,$$ 

    where $\Sigma_0f=f$ for all $f\in F(\mathfrak{X})$.

    Suppose that $f\in \ker{\Sigma_k}$. Then 
    
    $$f=\Sigma_0f =(\Sigma_k Q(\Sigma_1)+\Sigma_\ell R(\Sigma_1))f,$$

    which reduces to 
    
    $$\Sigma_\ell R(\Sigma_1)f=f,$$ 
    
    since $f\in \ker{\Sigma_k}$. Let $g=R(\Sigma_1)f$. Then $g\in \ker{\Sigma_k}$ and $\Sigma_\ell g=f$, so that $\Sigma_\ell:\ker{\Sigma_k} \to \ker{\Sigma_k}$ is surjective. It follows by Proposition \ref{prop:linalg} that an $f$ exists satisfying $\Sigma_k f = G$ and $\Sigma_ \ell f = H$, which is equivalent to the desired result.

\end{proof}

\newpage 

\bibliographystyle{alpha}
\bibliography{Gonzalez}
\end{document}